\documentclass[12pt,a4paper,enabledeprecatedfontcommands]{scrartcl}
\usepackage[utf8]{inputenc}
\usepackage[T1]{fontenc}
\usepackage{lmodern}

\usepackage[british]{babel}

\usepackage{doi,url}
\usepackage[numbers]{natbib}
\usepackage{etoolbox,letltxmacro,xifthen,ifdraft} 
\usepackage[dvipsnames]{xcolor}
\usepackage{amsmath,amssymb,enumitem,bbm,mathrsfs,amsthm,aliascnt,mathtools}
\usepackage{graphicx,subcaption}
\usepackage[retainorgcmds]{IEEEtrantools}
\usepackage{hyperref} 
\hypersetup{final,hidelinks}
\usepackage[atend]{bookmark} 

\providecommand{\email}[1]{\href{mailto:#1}{\nolinkurl{#1}}}

\setlist[enumerate,1]{label={(\roman*)}}
\setlist[enumerate,2]{label={(\alph*)}}
\setlist[enumerate,3]{label={(\Roman*)}}

\newcommand{\newsstheorem}[2]{
  \newaliascnt{#1}{dummy}
  \newtheorem{#1}[#1]{#2}
  \aliascntresetthe{#1}
  \expandafter\def\csname #1autorefname\endcsname{#2}
}
\numberwithin{dummy}{section}
\theoremstyle{plain}
  \newsstheorem{theorem}{Theorem}
  \newsstheorem{proposition}{Proposition}
  \newsstheorem{corollary}{Corollary}
  \newsstheorem{lemma}{Lemma}
\theoremstyle{definition}
  \newsstheorem{definition}{Definition}
  \newsstheorem{example}{Example}
  \newsstheorem{assumption}{Assumption}
\theoremstyle{remark}
  \newsstheorem{remark}{Remark}

\newenvironment{eqnarr*}{\begin{IEEEeqnarray*}{rCl}}{\end{IEEEeqnarray*}\ignorespacesafterend}

\newcommand\RR{\mathbb{R}}

\newcommand\PP{\mathbb{P}}

\newcommand\EE{\mathbb{E}}

\newcommand\Indic[1]{\Ind_{\{#1\}}}

\newcommand\e{{\rm e}}

\newcommand\Ind{\mathbbm{1}}

\makeatletter \newcommand\mathof[1]{{\operator@font#1}} \makeatother

\newcommand\dd{\mathof{d}}

\DeclarePairedDelimiterX\ip[2]{\langle}{\rangle}{#1,#2}
\begin{document}

\title{On a Feynman-Kac approach\\ to growth-fragmentation semigroups \\and their asymptotic behaviors}
\author{Jean Bertoin\footnote{Institute of Mathematics, University of Zurich, Switzerland}  }
\date{\unskip}
\maketitle 
\thispagestyle{empty}

\begin{abstract}
  \noindent This work develops  further a probabilist approach to the asymptotic behavior of growth-fragmentation semigroups  via the Feynman-Kac formula, which was introduced in a joint article with A.R. Watson \cite{BW}. Here, it is first shown that the sufficient condition for a Malthusian behavior which was established in \cite{BW}, is also necessary. We then provide a simple criterion to ensure exponential speed of convergence, which enables us to treat cases than were not covered previously in the literature.  
    
\end{abstract} 
{\small 
}

\section{Introduction}\label{s:intro}
Imagine a population, for instance of cells or of bacterias, where individuals grow and divide as time passes, and such that the evolution of each individual only depends on its own mass, without interaction between different individuals. Assume also that when a division event happens, the sum of the masses of daughters resulting from the division equals the mass of the mother before division. In other words, the total mass is a preserved quantity when division occurs, but may grow between consecutive division events. 
Growth-fragmentation equations provide a mathematical model for such dynamics, by describing the evolution of concentrations of individuals as a function of masses and time. The rate of growth of an individual may depend on its mass, and the rate at which a mother produces daughters 
 may also depend both on the mass of the mother just before division and on the masses of its daughters right after the division.
 
Specifically, one considers an operator of the form
\begin{equation}\label{e:A}
  \mathcal{A}f(x) = c(x)f'(x) + \int_0^x f(y) k(x,y) \dd y - K(x) f(x), \qquad x>0,  
\end{equation}
which is defined on some domain ${\mathcal D}_{\mathcal A}$ of  smooth functions $f: (0,\infty)\to \RR$. 
Here $c(x)$ describes the growth rate as a function of the mass, and 
 $k(x,y)$ the rate at which a daughter particle 
with mass $y$ appears as the result of the division of a mother particle
with mass $x>y$. Finally, $K(x)$ is the total rate of division of individuals
with mass $x$, and the assumption of conservation of mass
at division events thus translates into
\begin{equation}\label{e:cl}
 K(x)=\int_0^x  \frac{y}{x} k(x,y)\dd y.
\end{equation}

Under fairly simple general assumptions on the rates $c$ and $k$ that will be introduced later on, ${\mathcal A}$ is the infinitesimal generator of a strongly continuous positive semigroup $(T_t)_{t\geq 0}$, so that a growth-fragmentation equation can be given in the form
$$\frac{\dd T_tf}{\dd t}= {\mathcal A}T_tf\,,\qquad f\in {\mathcal D}_{\mathcal A}.$$
In this setting, the measure $\mu_t(x,\dd y)$ on $(0,\infty)$ such that 
$$T_tf(x) = \int_{(0,\infty} f(y)\mu_t(x, \dd y)\coloneqq \ip{\mu_t(x, \cdot)}{f}$$
describes the concentration at time $t$  of individuals of mass $y$, when one starts at time $0$ from a unit concentration of individuals of mass $x$, i.e. $\mu_0(x, \dd y)=\delta_x(\dd y)$. 

In general, there is of course no explicit expression for the growth-fragmentation semigroup $(T_t)_{t\geq0}$, and many works in this area are concerned with its large time asymptotic behavior. See in particular \cite{BCG-fine,  BPR12, BerGab, CCM11, DoumGab, Gabriel, Mic06, MMP05, MS16, PertRyz}  and further  references therein. Typically, one expects that under adequate  assumptions  on the rates of growth and of fragmentation, there exists a principal eigenvalue $\rho\in\RR$ 
such that 
\begin{equation}\label{e:paradigm}
\lim_{t\to \infty} \e^{-\rho t} T_tf(x) = h(x)\ip{\nu}{f},\qquad x>0,
\end{equation}
at least for every continuous and compactly supported function $f: (0,\infty)\to \RR$.
Here, $\nu(\dd y)$ is a Radon measure on $(0,\infty)$, which is often referred to as the asymptotic profile, and $h$ some positive function. We stress that \eqref{e:paradigm} may fail; see for instance \citet{DouEsc} and \citet{Gabriel}.

When  \eqref{e:paradigm} holds, it is further important to be able to estimate the speed of convergence. Indeed, say for $\rho >0$, 
an indefinite exponential growth is of course unrealistic in practice, and the growth-fragmentation equation can only be pertinent 
for describing rather early stages of the evolution of a population when certain effects such as competition between individuals  for space or resources can be neglected. As a consequence, the notions of principal eigenvalue and of asymptotic profile are only relevant for applications when the convergence in \eqref{e:paradigm}  occurs fast enough.

Spectral theory for semigroups and generators yields a well-established and classical framework for establishing the validity of  \eqref{e:paradigm}, again provided that the growth and fragmentation rates are properly chosen. In short, if one can find positive eigenelements, namely a Radon measure $\nu$ and a positive function
$h$ on $(0,\infty)$, such that for some $\rho\in \RR$:
$${\mathcal A}h=\rho h\ ,\ {\mathcal A}^*\nu= \rho \nu, \ \text{ and } \ip{\nu}{h}=1,$$
where ${\mathcal A}^*$ denotes the dual of ${\mathcal A}$, then the so-called general relative entropy method (see in particular Chapter 6 in Perthame \cite{Per07} and  \citet{MMP05}) shows that 
\eqref{e:paradigm} holds.  
In turn, explicit criteria in terms of the rates of growth $c$ and of fragmentation $k$  that ensure the existence of positive eigenelements, have  been obtained by Michel \cite{Mic06} and by Doumic and Gabriel \cite{DoumGab}. These works rely crucially on  the Krein-Rutman theorem, a version of the Perron-Frobenius theorem for positive compact operators.
On the other hand, exponential rate of convergence in \eqref{e:paradigm} is essentially  equivalent to  the existence of a spectral gap. This has been obtained under specific assumptions on the growth and fragmentation rates notably by 
\citet{PertRyz}, \citet{LauPert}, \citet{CCM11} and \citet{MS16}.

Quite recently, together with A.R. Watson \cite{BW}, we devised  a probabilistic approach to  \eqref{e:paradigm}, which circumvents spectral theory of semigroups and further provides probabilistic expressions for the various quantities of interest.
This requires some assumptions on the growth rate $c$ and the fragmentation $k$ that we now introduce.
First,
\begin{equation}\label{e:c-bound}
 \text{ the function $x\mapsto c(x)/x$ is continuous, positive and bounded on $(0,\infty)$,}
\end{equation}
and second, writing $\bar k(x,y)\coloneqq x^{-1}yk (x,y)$ for every $0<y<x$, 
\begin{equation}\label{H8}
\text{the map $ x\mapsto \bar k(x,\cdot)$ from $(0,\infty)$ to $L^1(\dd y)$ is continuous and bounded.}
\end{equation}

Our probabilistic  approach relies on an instrumental Markov process $X=(X_t)_{ t\geq 0}$ with infinitesimal generator
\begin{equation}\label{e:genMark}
{\mathcal G}f(x) \coloneqq c(x)f'(x)+ \int_0^x(f(y)-f(x)) \bar k(x,y) \, \dd y. 
\end{equation}
Assumption \eqref{H8} guaranties that the total jump rate \eqref{e:cl} remains bounded, so the jump times of $X$ never accumulate. One says  $X$ is piecewise deterministic (see \cite{refId0} and references therein), in the sense that 
the trajectory $t\mapsto X_t$ is driven by the steady flow velocity $c$ between two consecutive jumps, and jump times and locations are the sole source of randomness. We finally assume  that
\begin{equation}\label{e:irred}
  \text{the Markov process $X$ is irreducible,}
\end{equation}
that is, for every  $x,y>0$, the probability that the Markov process
started from $x$ visits $y>0$ is strictly positive. Roughly speaking, this means that there are no strict subintervals $I$ of $(0,\infty)$ that form traps for $X$, in the sense that once the path enters $I$, it cannot exit from $I$. 
Because $X$ is piecewise deterministic and has only downwards jumps, 
this can be ensured by a simple non-degeneracy assumption on the
fragmentation kernel $k$; see the forthcoming Lemma \ref{L8} and its proof for details.

The growth-fragmentation semigroup $T_t$
can then be given by a Feynman-Kac formula (we refer to \cite{Delmoral} for treatise on this topic in discrete time): \begin{equation}\label{e:fk}
  T_tf(x) = x \EE_x\left( \mathcal{E}_t \frac{ f(X_t)}{X_t}\right), \qquad \text{with }   \mathcal{E}_t\coloneqq \exp\left(\int_0^t \frac{c(X_s)}{X_s}  \dd s \right),
\end{equation}
where
$\EE_x$ stands for the 
expectation when $X$ starts at $X_0 = x$.
The first hitting time of $y>0$ by $X$, 
\begin{equation}\label{e:defH}
H(y)\coloneqq \inf\left\{ t>0: X_t=y\right\},
\end{equation}
and the Laplace transform
\begin{equation}\label{e:defL}
  L_{x,y}(q)
  \coloneqq
  \EE_x\left( \e^{-q H(y)}{\mathcal E}_{H(y)} , H(y)<\infty\right),
  \qquad q\in\RR,
\end{equation}
 then play a key role for the asymptotic behavior of $T_t$ as we shall now explain.
 
  Note that $L_{x,y}$ is always a non-increasing convex function with values in $(0,\infty]$, with $\lim_{q\to \infty} L_{x,y}(q)= 0$ and $\lim_{q\to -\infty} L_{x,y}(q)= \infty$. In particular, it possesses a right-derivative 
 $L'_{x,y}(q)$ at every point $q$ of its effective domain, i.e. with $L_{x,y}(q)<\infty$. Defining the Malthus exponent by
 \begin{equation} \label{e:lambda}
    \lambda\coloneqq \inf\{q\in\RR: L_{x,x}(q) < 1\}
\end{equation}
  (actually, this definition does not depend on $x>0$),
 the main results of  \cite{BW} can be summarized as follows. First,  if 
 \begin{equation}\label{e:condBW}
  L_{x,x}(\lambda)=1\ \text{ and }\  -L'_{x,x}(\lambda)<\infty
  \end{equation}
(again, this condition does not depend on $x$),  then  \eqref{e:paradigm} holds with $\rho=\lambda$. Moreover, the asymptotic profile $\nu$ and the function $h$ are given for some arbitrarily chosen  $x_0>0$  by
\begin{equation}\label{e:profile}
h (y) = y L_{y,x_0}(\lambda) \ \text{ and } \ 
\nu(\dd y)=  \frac{\dd y}{h(y) c(y) |L'_{y,y}(\lambda)|} , \qquad y>0.
\end{equation}
Second, if  
\begin{equation}\label{e:condBW2}
\text{there exists some }q<\lambda \text{ and  } x>0 \text{ with } L_{x,x}(q)<\infty,
  \end{equation}
  then  the convergence  \eqref{e:paradigm} takes place exponentially fast. Specifically,
  there exists $\beta >0$ such that 
\begin{equation}\label{e:expconv}
  \e^{-\lambda t} T_tf(x)=h(x)\ip{\nu}{f}  + o( \e^{-\beta t} ) \qquad \text{ as }t\to \infty
 \end{equation}
 for every continuous function $f$ with compact support and every $x>0$.
We stress that, by convexity of $L_{x,x}$, \eqref{e:condBW2} is of course a stronger requirement than \eqref{e:condBW}. 

Throughout the rest of this work, we always assume that \eqref{e:c-bound}, \eqref{H8} and \eqref{e:irred} hold.
We have two main purposes. First, we shall observe that the condition \eqref{e:condBW} is also necessary for the Malthusian behavior \eqref{e:paradigm}, and in particular, whenever the latter holds, the principal eigenvalue $\rho$  is always given by 
the Malthus exponent defined by \eqref{e:lambda}. We shall actually establish an even slightly stronger result.

\begin{theorem}\label{T1} 
Suppose that for some $\rho\in \RR$:
\begin{enumerate}
\item there exist $x_1>0$ and a continuous function $f:(0,\infty)\to \RR_+$ with compact support and $f\not\equiv 0$, such that
$$\limsup_{t\to \infty} \e^{-\rho t}T_tf(x_1)<\infty,$$

\item there exist $x_2>0$ and a continuous function $g:(0,\infty)\to \RR_+$ with compact support, such that
$$\liminf_{t\to \infty}\e^{-\rho t} T_tg(x_2)>0.$$
\end{enumerate}
Then $\rho = \lambda$, \eqref{e:condBW} holds and thus also the Malthusian behavior \eqref{e:paradigm}
with \eqref{e:profile}.
\end{theorem}

The assumptions in \eqref{e:condBW} and \eqref{e:condBW2} are given in terms of the Laplace transform $L_{x,y}$
rather than directly in terms of the coefficients $c$ and $k$ as one might have wished, and the second purpose of the present work is to remedy (at least partly) this problem by providing the following simple criterion.

\begin{theorem}\label{T2} 
If the Malthus exponent $\lambda$ defined in \eqref{e:lambda}  and  the growth rate $c$ fulfill
\begin{equation}\label{e:ccbis}
 \limsup _{x\to 0+} \frac{c(x)}{x}< \lambda \quad\text{and}\quad  \limsup _{x\to \infty} \frac{c(x)}{x}<\lambda,
 \end{equation}
then the exponential convergence \eqref{e:expconv} holds. 
\end{theorem}
Theorem \ref{T2} might seem unsatisfactory, as its requirements are not given only in terms of the rates $c$ and $k$, but also involve the Malthus exponent $\lambda$. However, there are simple explicit conditions in terms of $c$ and $k$ only that ensure \eqref{e:ccbis}. In particular, it is easily seen that $\lambda> \inf_{x>0} c(x)/x$ when $X$ is recurrent and $c$ is not linear, cf. Proposition 3.4(ii) in \cite{BW}. Thus \eqref{e:ccbis} is then fulfilled whenever
\begin{equation}\label{e:ccter} \lim_{x\to 0+} \frac{c(x)}{x}=\lim_{x\to \infty} \frac{c(x)}{x}= \inf_{x>0} \frac{c(x)}{x}.
\end{equation}
In turn, explicit conditions in terms of $c$ and $k$ guarantying recurrence for $X$ are easy to obtain, as we shall further discuss in Section \ref{s:MC}(i-ii). This yields explicit criteria for \eqref{e:expconv} that enables us to treat cases than were not covered previously in the literature. 

It may be interesting to discuss a bit further Criterion \eqref{e:ccbis}. Requesting an upper-bound for the growth rate at infinity should not come as a surprise as similar assumptions are made in the literature to prevent the formation of too large particles. For instance, \citet{DoumGab} request (among other assumptions) that $\lim_{x\to \infty} xK(x)/c(x)=\infty$, which forces in our setting $\lim_{x\to \infty} c(x)/x=0$ since we also assumed in \eqref{H8} that the total rate of fragmentation $K$ remains bounded; see Equation (13) in \cite{DoumGab}, and also Equation (11) in \citet{BCG-fine}. On the other hand, imposing an upper-bound for the growth rate at $0+$ may be more surprising, as on the contrary, it is often assumed in the literature that the growth for small particles should be strong enough in order to prevent shattering (see notably Equation (11) in \cite{DoumGab} and Equation (10) in \cite{BCG-fine}). One might be further puzzled by the fact that the fragmentation rate $k$ does not appear explicitly in \eqref{e:ccbis}; however, the value of the Malthus exponent $\lambda$ depends of course both on $c$ and $k$. 

Let us also try to offer a rather informal interpretation of \eqref{e:ccbis}. The Feynman-Kac formula provides a representation of  the growth-fragmentation semigroup $(T_t)_{t\geq 0}$ in terms of a weighted particle $(X_t, {\mathcal E}_t)$, where $X_t$ is the location of the particle at time $t$ and ${\mathcal E}_t$ its weight. The weight thus increases at rate $c(x)/x$ when the particle is located at $x$, and 
in this setting, the Malthus exponent $\lambda$ can be interpreted as the long-time average rate of increase of the weight. Then \eqref{e:ccbis} means that the weight of the particle increases more slowly than on average when the particle is either close to $0$ or close to $\infty$. Informally, the particle has thus a more important contribution to the Feynman-Kac formula when it stays away from $0$ and from $\infty$, that is essentially when it remains confined in a compact interval. And it is precisely for processes staying in compact spaces that exponential ergodicity is expected. 

The rest of this article is organized as follows. The two theorems are established in the next two sections, where the main ideas of the proofs are sketched first. We also gather in Section \ref{s:MC} miscellaneous comments about Theorem \ref{T2}, notably discussing further the connection with earlier results in the literature. 

 We conclude this introduction by recalling that the Feynman-Kac functional ${\mathcal E}$ defined in \eqref{e:fk} is multiplicative, in the sense that for every $s,t\geq 0$, there is the identity
$${\mathcal E}_{t+s} = {\mathcal E}_t \times ( {\mathcal E}_s\circ \theta_t),$$
where ${\mathcal E}_s\circ \theta_t$ stands for the functional ${\mathcal E}_s$ evaluated for the shifted path $X\circ \theta_t=X_{t+\cdot}$. In the sequel, this basic property will be often used without specific mentions, notably in combination with the Markov property.

\section{Proof of Theorem \ref{T1}}\label{s:PT1}

The arguments for proving Theorem \ref{T1} belong to the same vein as in \cite{BW}, with the difference that the role of remarkable martingales there is rather played here by supermartingales. Specifically, we shall first establish some properties of the first hitting time $H(y)$ and of its Laplace transform $L_{x,y}$, which are then applied to introduce  supermartingales related to the Feynman-Kac formula. Then we shall use the latter and introduce another one-parameter family of (possibly defective) Markov process $Y^{(q)}$ by   probability tilting. This yields a more direct probabilistic representation the growth-fragmentation semigroup, and analyzing the behavior of $Y^{(q)}$ via the regeneration property at return times then readily yields the conclusion. 

We start by considering the motion $t\mapsto x(t)$ of a Lagrangian particle in the steady flow velocity  $c$, viz.
$$\dd x(t)=c(x(t))\dd t,$$
which governs the dynamics of the piecewise deterministic process $X$ between consecutive jump times, and introduce some notation in this setting that will be useful in several parts of this work. For $0<x<y$, denote by $s(x,y)$ the travel time from $x$ to $y$, that is 
$$x(s(x,y))=y \qquad \text{ when } x(0)=x.$$
Obviously $s(\cdot, \cdot)$ decreases in the first variable and increases in the second one.
Consider also the  event $\Lambda_{x,y}$ that process $X$ started at $x$ reaches $y$ before making any jump. Since  $K(z)$ is the total jump rate when the process is located at $z$, we have
\begin{equation} \label{e:nojump}
p(x,y)\coloneqq \PP_x(\Lambda_{x,y}) = \exp\biggl(-\int_0^{s(x,y)} K(x(t))\dd t \biggr)=  \exp\biggl(-\int_x^y\frac{K(z)}{c(z)}\dd z\biggr).
\end{equation}
 This is a positive quantity which increases with $x$ and decreases with $y$.

We proceed with the following uniform lower-bound for the cumulative distribution functions of first hitting times \eqref{e:defH}.

\begin{lemma}\label{L-1} For every $0<a<b$, there exists $t(a,b)\in\RR_+$ such that
$$\inf_{x,y\in[a,b]} \PP_x(H(y)<t(a,b))>0.$$
\end{lemma}
\begin{proof}    
  Consider first the process $X$ started from $b$.
The irreducibility assumption \eqref{e:irred} ensures that we can find two real numbers $q(a,b)\in(0,1)$ and $r(a,b)>0$
such that
$$\PP_b(H(a)<r(a,b))>q(a,b).$$

Next consider the process $X$ started from an arbitrary point $x\in [a,b]$. By focusing on trajectories which first hit $b$ before having any jump, then need an amount of time less than $r(a,b)$ for traveling from $b$ to $a$, and finally  hit $y\in[a,b]$ before having any further jump, we deduce from an application of the strong Markov property that there is the lowerbound
$$ \PP_x(H(y)<2s(a,b)+ r(a,b))> p(a,b)^2q(a,b)>0.$$
This proves our claim with $t(a,b)=2s(a,b)+ r(a,b)$.
\end{proof}

Next, recall the notation \eqref{e:defL} for the Laplace transform $L_{x,y}$,  \eqref{e:lambda} for the Malthus exponent, and fix $x_1>0$ arbitrarily.

\begin{lemma} \label{L1}
For every $q\geq \lambda$,  the function
$${\ell}_q: x\mapsto   L_{x,x_1}( q ),
  \qquad x>0,
$$
is  bounded away from $0$ and from $\infty$ on every compact interval of $(0,\infty)$.
\end{lemma}
\begin{proof}
Let us assume that $ q >0$, the case when $ q \leq 0$ being somewhat simpler. We have  plainly
$${\ell}_q(x)\geq v_q(x,x_1)\coloneqq  
 \EE_x\left( \e^{- q  H(x_1)}, H(x_1)<\infty\right).
$$
Fix $a>0$ arbitrarily small and $b>0$ arbitrarily large, with $a<x_1<b$. 
Lemma \ref{L-1} ensures the existence of $t(a,b)>0$ and $p>0$ such that, for all $x,y\in[a,b]$, 
$$v_q(x,y) \geq  \EE_x\left( \e^{- q  H(y)}, H(y)<t(a,b)\right) \geq p \e^{-  q  t(a,b)}.$$
 {\it A fortiori},  $\inf_{a\leq x \leq b} {\ell}_q(x)>0$.

On the other hand,  our assumption $q\geq \lambda$ and the very definition \eqref{e:lambda} entail that $L_{x_1,x_1}(q')<1$ for every $q'> q $. Equation (16) in \cite{BW} states that then
$L_{x,x_1}(q') L_{x_1,x}(q') <1$ for every $x>0$. By right-continuity of the functions $L_{x,y}$, we have 
$${\ell}_q(x)=L_{x,x_1}( q ) \leq 1/L_{x_1,x}( q ),$$
and  since 
 $$L_{x_1,x}( q )\geq v_q(x_1,x)\geq p \e^{-  q  t(a,b)},$$
 we conclude that $\sup_{a\leq x \leq b} {\ell}_q(x)<\infty$.
\end{proof}
 
Theorem 4.4  in \cite{BW}, which states that if  $L_{x,x}(\lambda)=1$, then the process $\left( \e^{-\lambda t}{\ell}_{\lambda}(X_t) {\mathcal E}_t\right)_{t\geq 0}$ is a martingale, is a cornerstone of the probabilistic approach which is developed there.  
Here is a version of the latter in terms supermartingales.

\begin{lemma} \label{L2} For every $q\geq \lambda$, 
 the process 
$${\mathcal S}^{(q)}_t\coloneqq \e^{- q  t}{\ell}_q(X_t) {\mathcal E}_t, \qquad t\geq 0$$
is a $\PP_x$-supermartingale for every $x>0$. 
\end{lemma}

\begin{proof}
 Write 
$$N_t\coloneqq \#\{0<s\leq t: X_s=x_1\},\qquad t\geq 0$$
for the process which counts the number of visits of $X$ to $x_1$ as time passes, and 
$$R_n\coloneqq \inf\{t>0: N_t=n\}, \qquad n\geq 1$$
for the instant when $X$ returns to $x_1$ for the $n$-th time.
Write also $({\mathcal F}_t)_{t\geq0}$ for the natural filtration of $X$ and recall that the return times $R_n$ are 
$({\mathcal F}_t)$-stopping times. Further, writing ${\mathcal G}_n\coloneqq {\mathcal F}_{R_n}$, we know that 
for every $t\geq 0$, $N_t+1$ is a $({\mathcal G}_n)$-stopping time and the first return to $x_1$ after time $t$ can be expressed as 
$$D_t\coloneqq \inf\{s>t: X_s=x_1\} = R_{N_t+1}. $$

On the one hand, we see from the Markov property at time $t$ and the definition of the function $\ell_q$ in Lemma \ref{L1}  that for every $x>0$, 
\begin{equation}\label{e:towerprep}
{\mathcal S}^{(q)}_t =\EE_x\left( \e^{- q  D_t} {\mathcal E}_{D_t}\Indic{D_t<\infty} \mid {\mathcal F}_t\right)= \EE_x\left( \e^{- q  {R_{N_t+1}}} {\mathcal E}_{R_{N_t+1}} \Indic{{R_{N_t+1}}<\infty} \mid {\mathcal F}_t\right).
\end{equation} 
On the other hand, the strong Markov property and the fact that $L_{x_1, x_1}( q )\leq 1$ (from the definition \eqref{e:lambda} and our assumption $q\geq \lambda$) entail that for every $x>0$,
$$\e^{- q  R_n} {\mathcal E}_{R_n}\Indic{R_n<\infty}\,, \qquad n\geq 1$$
is a $\PP_x$-supermartingale in the filtration $\left({\mathcal G}_n\right)_{n\geq 1}$.
Since for $s\leq t$, $N_s+1\leq N_t+1$ are two $({\mathcal G}_n)$-stopping times, it follows from the optional sampling theorem for nonnegative supermartingales that 
$$\EE_x\left( \e^{- q  {R_{N_t+1}}} {\mathcal E}_{R_{N_t+1}} \Indic{{R_{N_t+1}}<\infty} \mid {\mathcal G}_{N_s+1}\right)
\leq \e^{- q  {R_{N_s+1}}} {\mathcal E}_{R_{N_s+1}} \Indic{{R_{N_s+1}}<\infty}.
$$
Then on both sides,  take the conditional expectation given ${\mathcal F}_s$, which is a sub-algebra of ${\mathcal F}_{D_s}= {\mathcal G}_{N_s+1}$. We get from \eqref{e:towerprep} (applied at time $s$ rather than $t$)
$$\EE_x\left( \e^{- q  {R_{N_t+1}}} {\mathcal E}_{R_{N_t+1}} \Indic{{R_{N_t+1}}<\infty} \mid {\mathcal F}_s\right)
\leq {\mathcal S}^{(q)}_s.
$$
We conclude the proof by using once again \eqref{e:towerprep} and the so-called tower property of conditional expectations on the left-hand.
\end{proof}

The supermartingale ${\mathcal S}^{(q)}$ in  Lemma \ref{L2} enables us to introduce a possibly defective (i.e. possibly with finite lifetime $\zeta$)   c\`adl\`ag Markov process $Y^{(q)}=(Y^{(q)}_t)_{0\leq t < \zeta}$ with distribution denoted by  $\PP^{(q)}$ as follows. For every  $t\geq 0$ and every nonnegative functional $F$ defined on  Skorokhod's space ${\mathcal D}_{[0,t]}$ of c\`adl\`ag paths $\omega: [0,t]\to (0,\infty)$, one sets 
\[ \EE^{(q)}_x[ F((Y^{(q)}_{s})_{0\leq s \leq t}), \zeta>t]
  = \frac{1}{{\ell}_q(x)} \EE_x[ \mathcal{S}^{(q)}_t F((X_{s})_{0\leq s \leq t})], \qquad x > 0.
\]
We stress that 
 the distribution of $(Y^{(q)}_{s})_{0\leq s \leq t}$ under the conditional law 
$\PP^{(q)}_x(\cdot \mid \zeta>t)$ is absolutely continuous with respect to that $(X_{s})_{0\leq s \leq t}$ under $\PP_x$, and as a consequence, $Y^{(q)}$ inherits irreducibility from \eqref{e:irred}.

\begin{lemma} \label{L3} \begin{enumerate}
\item[(i)] Suppose that the assumption (i) of Theorem \ref{T1} holds.
Then $ \rho\geq \lambda$.
\item[(ii)] Suppose further that the assumption (ii) of Theorem \ref{T1} also holds, and set $Y\coloneqq Y^{(\rho)}$. 
 Then  there exists $b>0$ and $s>0$ sufficiently large, such that
 $$\liminf_{t\to \infty} \PP^{(\rho)}_{b}\left( Y_r=b \text{ for some } r\in[t,t+s]\right)>0.$$
 \end{enumerate} 
 
\end{lemma}

\begin{proof}
(i) Suppose $\rho<\lambda$ and  pick any $q\in(\rho,\lambda)$. On the one hand, since $q>\rho$, assumption (i) of Theorem \ref{T1} entails that
$$\int_0^{\infty} \e^{-q t} T_tf(x_1)\dd t <\infty$$
for some $x_1>0$ and some continuous function $f: (0,\infty)\to \RR_+$ with $f\not\equiv 0$. On the other hand, 
as $q<\lambda$, we have $L_{x_1, x_1}(q)\geq 1$, and the assertion above contradicts Proposition 3.3 in \cite{BW}.

(ii)  Thanks to (i), we may now take $q=\rho$. 
Note from the very definition of $Y$ that the Feynman-Kac formula \eqref{e:fk} can be translated as follows: for every $x>0$ and every continuous and compactly supported function $f:(0,\infty)\to \RR$, there is the identity 
\begin{equation} \label{e:FK2}
\frac{\e^{-\rho t}}{x {\ell}(x)} T_tf(x)=  \EE_x^{(\rho )}\left( \frac{ f(Y_t)} {Y_t \ell(Y_t)},  \zeta >t \right),
\end{equation}
with $\ell\coloneqq {\ell}_{\rho}$.
Combining this with assumption (ii) of Theorem \ref{T1} and the fact that, thanks to Lemma \ref{L1},  $\ell$ remains bounded away from $0$ on compact intervals of $(0,\infty)$, we deduce that
\begin{equation} \label{e:probmin}
\liminf_{t\to \infty} \PP^{(\rho)}_{x_2}(Y_t\in[a,b], \zeta>t) >0,
\end{equation}
for any  $0<a<b$ such that  ${\rm Supp}(g)\subseteq [a,b]$.

On the other hand, recall the notation from the first paragraph in the proof of Lemma \ref{L-1}, and 
for every $x\in[a,b]$, consider the probability $p^{(\rho)}(x)$ that process $Y$ started at $x$ reaches $b$ at time $s(x,b)$, before dying or making any jump.  The  obvious bound  $ {\mathcal S}^{(\rho)}_t\geq \ell(X_t)\exp(-\rho^+ t)$ (where $\rho^+$ stands for the positive part of $\rho$) yields
\begin{eqnarray*}p^{(\rho)}(x)&=&\frac{1}{\ell(x)} \EE_x\left({\mathcal S}^{(\rho)} _{s(x,b)}, X \text{ has no jump before time }s(x,b)\right)\\
&\geq &\frac{ \ell(b)}{\ell(x)} \exp\left( -\rho^+ s(x,b)\right) p(x,b)\\
&\geq &\frac{\ell(b)}{\sup_{[a,b]}\ell} \exp\left( -\rho^+ s(a,b)\right)  p(a,b).
\end{eqnarray*}
Again by Lemma \ref{L1} and  the first paragraph in the proof of Lemma \ref{L-1}, the right-hand above is positive, hence
$$\inf_{[a,b]}p^{(\rho)}>0.$$

We now see from the Markov property of $Y$ that
$$\PP^{(\rho)}_{x_2}\left( Y_r=b \text{ for some } r\in[t,t+s(a,b)]\right)\geq \inf_{[a,b]}p^{(\rho)}\times \PP^{(\rho)}_{x_2}(Y_t\in[a,b], \zeta>t),$$
and then, combining with \eqref{e:probmin}, that
$$\liminf_{t\to \infty} \PP^{(\rho)}_{x_2}\left( Y_r=b \text{ for some } r\in[t,t+s(a,b)]\right)>0.$$
Recalling that $Y$ is irreducible and applying the strong Markov property at time $H_Y(b)$ completes the proof. 
\end{proof}

We readily deduce from Lemma \ref{L3} the following
\begin{corollary}\label{C1} Under the assumptions (i) and (ii) of Theorem \ref{T1}, the Markov process $Y=Y^{(\rho)}$ is point-recurrent and positive, that is
$$\EE_x(H_Y(y))<\infty\qquad \text{ for every }x,y>0,$$
where $H_Y(y)\coloneqq \inf\{t\in(0,\zeta): Y_t =y\}$ stands for the first hitting time of $y$ by the process $Y$, with the usual convention that $\inf \emptyset =\infty$.
\end{corollary} 
\begin{proof}
We start by observing from Lemma \ref{L3}(ii), the irreducibility of $Y$, and an application of the strong Markov property at time $H_Y(b)$ that $b$ must be a recurrent state for $Y$, i.e. $$\PP_b(H_Y(b)<\infty)=1.$$ By the strong Markov property at return times, the set of passage times at $b$, $\{t\geq 0: Y_t=b\}$ is regenerative, i.e. it can be expressed as the set of partial sums of a sequence of independent copies of the variable $H_Y(b)$. 
Its so-called renewal function is given by
$$U(t)\coloneqq \EE_b\left({\rm Card}\{r\leq t: Y_r=b\}\right),\qquad t\geq 0,$$
and the elementary renewal theorem shows that
$$\lim_{t\to \infty} t^{-1}U(t) = 1/\EE_b(H_Y(b)).$$
Again from an application of the strong Markov property at the first hitting time of $b$ under $\PP_{x_2}^{(\rho)}$, we easily see that  for every $s>0$, there is the lower bound
$$\lim_{t\to \infty} t^{-1}U(t) \geq s^{-1} \liminf_{t\to \infty} \PP^{(\rho)}_{x_2}\left( Y_r=b 
\text{ for some } r\in[t,t+s]\right),$$
and we now can deduce from Lemma \ref{L3}(ii) that $\EE_b(H_Y(b))<\infty$. 
Since $Y$ is irreducible, the conclusion of the statement follows. 
\end{proof}

We now have all the ingredients needed to prove Theorem \ref{T1}.
Indeed, Corollary \ref{C1} implies that under the assumptions (i) and (ii) of Theorem \ref{T1},
$Y$ cannot be defective, i.e. $\PP_x^{(\rho)}(\zeta =\infty)=1$. That is, equivalently, $\EE_x({\mathcal S}^{(\rho)}_t)=\ell(x)$ for all $t\geq 0$, and since a supermartingale with a constant expectation must be a martingale, ${\mathcal S}^{(\rho)}$ is a $\PP_x$-martingale for every $x>0$. 

Then, point-recurrence for $Y$ gives for every $x>0$
$$1= \lim_{t\to \infty} \PP^{(\rho)}_x(H_Y(x) \leq t) = \lim_{t\to \infty}
\frac{1}{\ell(x)} \EE_x[ \mathcal{S}^{(\rho)}_t, H (x) \leq t].$$
On the other hand, the martingale property of ${\mathcal S}^{(\rho)}$ under $\PP_x$ and the optional sampling theorem yield
$$\EE_x[ \mathcal{S}^{(\rho)}_t, H(x) \leq t]=\EE_x[ \mathcal{S}^{(\rho)}_{H(x)}, H(x) \leq t]=\ell(x)\EE_x[ \e^{-\rho H(x)} \mathcal{E}_{H(x)}, H(x) \leq t].
$$
We deduce by monotone convergence that 
$$L_{x,x}(\rho)= \EE_x[ \e^{-\rho H(x)} \mathcal{E}_{H(x)}, H(x) <\infty]=1,$$
which implies both that $\lambda=\rho$ and the first condition in \eqref{e:condBW} holds. 

We now see that the function $\ell=\ell_{\rho}=\ell_{\lambda}$ here is the same as that in Section 4 of \cite{BW}, the martingale $\mathcal{S}^{(\rho)}$ coincides with the martingale ${\mathcal M}$ there, and finally, the 
Markov process $Y$ here is the same as that in Section 5 of \cite{BW}. Recall from Corollary \ref{C1} that $Y$ is positive recurrent, and we conclude from Lemma 5.2(i) in \cite{BW} that the right-derivative of the Laplace transform $L_{x,x}(\cdot)$ at $\lambda$ is necessarily finite, which 
is the second condition in \eqref{e:condBW}. The proof of Theorem \ref{T1} is now completed.

\section{Proof of Theorem \ref{T2}}\label{s:PT2}

Our goal in this section is to check that, when the assumptions of Theorem \ref{T2} are fulfilled, then \eqref{e:condBW2} holds, as the exponential convergence then follows from Theorem 1.1  in \cite{BW}. 
This will be achieved in three main steps.

To start with, we work on a compact interval $[a,b]$, where $0<a<b$ are given, and consider the first exit-time 
$$\sigma(a,b)\coloneqq \inf\{t>0: X_t\not\in [a,b]\}.$$
We first discuss  irreducibility for the process killed when exiting from $[a,b]$,
which is a necessary preamble for the rest of our analysis. We then verify that the Krein-Rutman theorem can be applied  in this compact setting, by analyzing the trajectories of $X$. 
This yields a principal eigenvalue $\rho_{a,b}$ for the system where particles are killed when exiting $[a,b]$, and a corresponding positive eigenfunction $h_{a,b}$. We then construct useful martingales  from the latter, which in turn will enable us to compute certain expectations by application of optional sampling.

For the next step, we fix the lower-boundary point $a$ small enough (respectively, the upper-boundary point $b$ large enough), and let $b$ tend to $\infty$ (respectively, $a$ tend to $0+$). We shall establish the existence of $\lambda_a<\lambda$  and a non-degenerate function $g_a: [a,\infty)\to \RR_+$ (respectively, $\lambda^b<\lambda$ and $g^b: (0,b]\to \RR_+$) such that
the process $ g_a(X_t) {\mathcal E}_t \e^{-\lambda_a t} \Indic{t<\sigma(a,\infty)}$ (respectively, 
$ g^b(X_t) {\mathcal E}_t \e^{-\lambda^b t} \Indic{t<H(b)}$) is a supermartingale.

Finally, \eqref{e:condBW2} is established by putting the pieces together. In short, we pick a large enough interval $[a,b]$, $q<\lambda$ close enough to $\lambda$, decompose the excursion of the process away from its starting point 
at certain first-exit times, and
estimate the various pieces using the preceding steps.

\subsection{Irreducibility in compact intervals}
We start by addressing the slightly technical question of irreducibility. Even though $X$ has been assumed to be irreducible, it may happen that for some $0<a<b$, there exist two states $x<y$ both in $ (a,b)$ such that no path of $X$ started from $y$ can reach $x$ without exiting first from $[a,b]$. Recall however that  the probability that $X$
started from $x$ follows the flow velocity without having jumps until it reaches $y$ is always positive, so the problem can only arise when the starting point is larger than the target. When this occurs, the process killed at time $\sigma(a,b)$ is then no longer irreducible, and this creates an obstacle for our analysis. 

We call an interval $(a,b)$ with $0<a<b$ good, if the process killed at time $\sigma(a,b)$ remains irreducible, that is if  \begin{equation}\label{e:kilirr}
\PP_x(H(y)<\sigma(a,b))>0\qquad \text{ for all } x,y\in(a,b).
\end{equation}
We now argue that we can always find good intervals $(a,b)$ with  $a>0$ as small as we wish and $b$ as large as we wish.

\begin{lemma}\label{L8} For every $\varepsilon \in(0,1)$, there exists a good interval $(a,b)$ with $a<\varepsilon$ and $b>1/\varepsilon$.
\end{lemma}

\begin{proof}
Consider any $\beta>0$ such that total jump rate $K(\beta)>0$, and then any $\alpha<\beta$ such that the right-neighborhood of $\alpha$ belongs to the support of $\bar k(\beta, \cdot)$, i.e. $\int_\alpha^x\bar k(\beta,y)\dd y>0$ for all $x>\alpha$. Thanks to \eqref{H8}, the same still holds when we replace $\beta$ by any $\beta'<\beta$ close enough to $\beta$, and using again the fact that on any finite time interval, the probability that  $X$ follows the flow velocity without having jump is positive, we now readily see that \eqref{e:kilirr} holds for $a=\alpha$ and $b=\beta$.

Such intervals $(\alpha,\beta)$ form a covering of $(0,\infty)$, as otherwise, the assumption of irreducibility \eqref{e:irred} would fail. There thus exists a finite covering of the compact interval $[\varepsilon, 1/\varepsilon]$, say $\{(\alpha_i, \beta_i): i=1, \ldots, n\}$, and it is then easy to check from the strong Markov property  that $a=\min \alpha_i$ and $b=\max \beta_i$ fulfill the requirements of the statement.  \end{proof}

\subsection{Applying the Krein-Rutman theorem in a compact interval}
We next
consider the Banach space ${\mathcal C}_0[a,b)$ of continuous functions $f: [a,b]\to \RR$ with $f(b)=0$, endowed with the usual norm $\|f\|=\sup_{x\in [a,b)}|f(x)|$. The reason for imposing  $f(b)=0$  is that the two boundary points  the interval $[a,b]$ have a different status for the Markov process killed at time $\sigma(a,b)$. Specifically, $a$ is an entrance boundary, in the sense that the process started at $a$ then stays in $[a,b]$ for a strictly positive amount of time $\PP_a$-a.s., whereas $b$ is an exit boundary, meaning that the process started at $b$ leaves $[a,b]$ instantaneously $\PP_b$-a.s. We do not assume right now that the interval $(a,b)$ is good, but this assumption will of course be essential in a later part of our analysis.

Recall our assumption \eqref{e:c-bound} and define 
$q_c\coloneqq 1+\sup_{x>0} c(x)/x$, so that
$$ {\mathcal E}_t \e^{-t q_c} \leq  \e^{-t} \qquad \text{ for all }t\geq 0.$$
 We
introduce for every bounded measurable function $f:[a,b]\to \RR$
$$U_{a,b}f(x)\coloneqq \EE_x\left(\int_0^{\sigma(a,b)} f(X_t) {\mathcal E}_t \e^{-t q_c} \dd t\right),\qquad x\in[a,b].$$

\begin{lemma} \label{L4}
The operator $U_{a,b}$ maps ${\mathcal C}_0[a,b)$ into itself. More precisely
the family of functions
$\{U_{a,b}f:  \| f\|\leq 1\}$ is equicontinuous.

\end{lemma} 
\begin{proof} We first note that  $U_{a,b}f(b)=0$  (since $\sigma(a,b)=0$, $\PP_b$-a.s.), and also that 
$U_{a,b}$ is a contraction, i.e. $\| U_{a,b}f\| \leq \| f \|$.
Then recall that for $0<x<y$, $s(x,y)$ denotes the travel time from $x$ to $y$ for a Lagrangian particle driven by the flow velocity $c$, and that $\Lambda_{x,y}$ stands for the event that  $X$ starts from $x$ and reaches $y$ before making any jump. Observe that on that event, we have
$$\ln {\mathcal E}_{s(x,y)} = \int_0^{s(x,y)} \frac{c(x(s))}{x(s)}\dd s = \ln y-\ln x.$$

Take $a\leq x <y \leq b$. By 
decomposing the trajectory at time $s(x,y)$ and applying the Markov property on the event $\Lambda_{x,y}$, we now easily see that for every $ f$ with $ \| f\|\leq 1$, there is the  inequality 
$$|U_{a,b}f(x)-U_{a,b}f(y)| \leq 2(1- \PP_x(\Lambda_{x,y})) +s(x,y) + \left| \frac{y}{x}\e^{-q_cs(x,y)} -1\right|.
$$
On the one hand,  \eqref{e:nojump} shows that $1-\PP_x(\Lambda_{x,y})$ converges to $0$ as $y-x\to 0+$, uniformly for $a\leq x < y \leq b$. On the other hand, it is easily checked that the same holds for $s(x,y)$ (because the flow velocity $c$ is bounded away from $0$ on $[a,b]$). 
This entails that 
$$ \sup_{x,y\in[a,b], |y-x|<\varepsilon} |U_{a,b}f(x)-U_{a,b}f(y)| \to 0 \quad \text{as } \varepsilon \to 0+, $$
uniformly for $f$ with $ \| f\| \leq 1$, and our claim is proven.
\end{proof}

 The subspace ${\mathcal C}^+_0[a,b)$ of nonnegative functions in ${\mathcal C}_0[a,b)$ is a reproducing cone, that is ${\mathcal C}^+_0[a,b)$ is a closed convex set which is stable by multiplication by nonnegative constants, and $f=f^+-f^-$ is a decomposition of a generic function $f\in {\mathcal C}_0[a,b)$ as the difference of two functions in ${\mathcal C}^+_0[a,b)$. We stress that, due to the $0$ boundary condition at $b$, the interior of ${\mathcal C}^+_0[a,b)$ is empty and ${\mathcal C}^+_0[a,b)$ is not a solid cone.
 
 It follows from Lemma \ref{L4} by the Arzel\`a-Ascoli theorem that
the operator $U_{a,b}$ is  compact. Obviously, it is also positive, i.e. maps the cone ${\mathcal C}^+_0[a,b)$ into itself. However, since ${\mathcal C}^+_0[a,b)$ is not a solid cone, $U_{a,b}$ is not strongly positive, and
 in order to apply the Krein-Rutman theorem to $U_{a,b}$, we thus still need to establish positivity of its spectral radius. 

\begin{lemma} \label{L6} For every good interval $(a,b)$, the spectral radius  $r(a,b)$ of $U_{a,b}$ is positive.
\end{lemma} 
\begin{proof} 
Let $U_{a,b}^n$ denote the $n$-th power of $U_{a,b}$, so that by Gelfand's formula, 
$$r(a,b)= \lim_{n\to \infty}\sup\{\|U_{a,b}^n f\|^{1/n}: f\in {\mathcal C}_0[a,b), \| f\| \leq 1\}.$$
Take any $a'<b'$ in $(a,b)$ and  
consider the function $f\in {\mathcal C}^+_0[a,b)$ with  $\|f\|=1$, such that $f\equiv 1$ on $[a,a']$, $f\equiv 0$ on $[b',b]$,
and $f(x)=(b'-x)/(b'-a')$ for $x\in[a',b']$. 
Since $(a,b)$ is good, $\PP_{b'}(H(a')<\sigma(a,b))>0$, and it follows from the strong Markov property applied at time $H(a')$ that 
$$U_{a,b}f(b')\geq \EE_{b'}\left( {\mathcal E}_{H(a')}\e^{-q_c H(a')}, H(a')<\sigma(a,b)\right) U_{a,b}f(a')> 0.$$ 

For every $x\in[a,b']$, by focusing on the event $\Lambda_{x,b'}$ where trajectories follow the steady flow velocity $c$ until the hitting time of $b'$ without having any jump, and then applying the strong Markov property at time $H(b')$, 
 we get the lowerbound 
$$U_{a,b}f(x) \geq \frac{b'}{x} \exp(-q_c s(x,b'))\PP_x(\Lambda_{x,b'}) U_{a,b}f(b').$$
Then, plainly,
$$\inf_{a\leq x \leq b'}  U_{a,b}f(x) \geq \gamma\coloneqq \exp(-q_c s(a,b'))\PP_a(\Lambda_{a,b'}) U_{a,b}f(b')>0,$$
and therefore we have $U_{a,b}f(x) \geq \gamma f(x)$ for all $x\in[a,b]$. Since $U_{a,b}$ is a positive operator, we conclude from Gelfand's formula that $\gamma$ is a lowerbound for the spectral radius, and {\it a fortiori}  $r(a,b)>0$.
\end{proof}
We have now checked all the requirements for the Krein-Rutman theorem (see, e.g. Chapter 6 in Deimling \cite{Deimling}), which asserts that the spectral radius $r(a,b)$ is then an eigenvalue of the operator $U_{a,b}$ and also of the dual operator $U_{a,b}^*$, and that the corresponding eigenfunctions can be chosen positive. 
In this direction,  recall from the Riesz-Markov representation theorem that 
any linear functional on ${\mathcal C}_0[a,b)$ which is positive (in the sense that it maps ${\mathcal C}^+_0[a,b)$ into $\RR_+$), can be represented by a finite Borel measure  on $[a,b]$ that has no atom at $b$.
Plainly, the dual operator $U_{a,b}^*$ maps any such measure, say $m$, to another finite measure $U_{a,b}^*m$ on
$[a,b]$ without  atom at $b$, via the identity 
$$\ip{U_{a,b}^*m}{f}= \ip{m}{U_{a,b} f}\qquad \text{for all } f\in{\mathcal C}^+_0[a,b).$$
This is the first milestone for the proof of Theorem \ref{T2}, and we record it for future use.

\begin{proposition} \label{P2} Let $(a,b)$ be a good interval. Then there exist a function $h_{a,b}\in {\mathcal C}^+_0[a,b)$ with $h_{a,b}(x)>0$ for every $x\in[a,b)$ and a finite measure $\nu_{a,b}$ on $[a,b]$ with $\nu_{a,b}(\{b\})=0$, such that
$$\ip{\nu_{a,b}}{h_{a,b}}=1\ ,\ U_{a,b} h_{a,b} = r(a,b) h_{a,b}\ , \text{ and } U_{a,b}^*\nu_{a,b} = r(a,b)\nu_{a,b}.$$
\end{proposition}
\begin{proof} We are only left with the proof of the positivity assertion for $h_{a,b}$ on $[a,b)$. But this immediately follows from the irreducibility \eqref{e:kilirr}, the identity 
$$h_{a,b}(x) = r(a,b) ^{-1} \EE_x\left(\int_0^{\sigma(a,b)} h_{a,b}(X_t) {\mathcal E}_t \e^{-t q_c} \dd t\right),$$
 and the fact that $h_{a,b}\in{\mathcal C}^+_0[a,b)$ is not identically $0$.
\end{proof}

Proposition \ref{P2} enables us to introduce the following  useful martingale.
\begin{lemma} \label{L7} Set $\rho_{a,b}\coloneqq  q_c-1/r(a,b)$. 
The process
$${\mathcal M}_{a,b}(t)\coloneqq \Indic{t<\sigma(a,b)} h_{a,b}(X_t) {\mathcal E}_t \e^{-t \rho_{a,b}}\,, \qquad t\geq 0$$
is a $\PP_x$-martingale for every $x\in[a,b]$.
\end{lemma} 
\begin{proof} Recall that $({\mathcal F}_t)_{t\geq 0}$ denotes the natural filtration of $X$. By the Markov property, we can express the martingale
$$M_t\coloneqq\EE_x\left(\int_0^{\sigma(a,b)} h_{a,b}(X_s) {\mathcal E}_s \e^{-s q_c} \dd s \, \mid \, {\mathcal F}_t\right)$$
as
$$M_t= \int_0^{t\wedge \sigma(a,b)} h_{a,b}(X_s) {\mathcal E}_s \e^{-s q_c} \dd s+ \Indic{t<\sigma(a,b)}  {\mathcal E}_t \e^{-t q_c} U_{a,b}h_{a,b}(X_t).$$
From the identity $U_{a,b}h_{a,b}=r(a,b) h_{a,b}$ and stochastic calculus, we deduce that 
$${\mathcal M}_{a,b}(t) = {\mathcal M}_{a,b}(0) + \frac{1}{r(a,b)} \int_0^t \e^{-s q_c} \dd M_s,$$
and the stochastic integral in the right-hand side is a $\PP_x$-martingale. 
\end{proof}

We now arrive at a second milestone of the proof of Theorem \ref{T2}.
\begin{proposition}\label{P1}
Take any good interval $(a,b)$. Then 
for  all $x,y\in(a,b)$, there is the identity
$$\EE_x\left({\mathcal E}_{H(y)} \e^{-\rho_{a,b}H(y)}, H(y) < \sigma(a,b)\right)=h_{a,b}(x)/h_{a,b}(y).$$
\end{proposition}
\begin{proof} Perhaps, it could be  tempting to try to derive the statement from Lemma \ref{L7} by an application of optional sampling to the martingale ${\mathcal M}_{a,b}$ and the stopping time $H(y)$. Note however that this would not be legitimate as the latter is not bounded and the martingale ${\mathcal M}_{a,b}$ is not uniformly integrable. 

We  use the martingale ${\mathcal M}_{a,b}$ to  introduce a new Markov process 
$\left(Z_t\right)_{t\geq 0}$ on $[a,b)$ with law ${\mathbf P}_x^{(a,b)}$, by setting
\[ {\mathbf E}^{(a,b)}_x[ F((Z_s)_{0\leq s \leq t})]
  = \frac{1}{h_{a,b}(x)} \EE_x[ \mathcal{M}_{a,b}(t) F((X_{s})_{0\leq s \leq t})],
\]
where $F$ stands for a generic nonnegative functional  defined on  Skorokhod's space ${\mathcal D}_{[0,t]}$.
Recall Proposition \ref{P2} and consider the probability measure $m_{a,b}(\dd x)\coloneqq h_{a,b}(x) \nu_{a,b}(\dd x)$ on $[a,b)$. We claim that $m_{a,b}$
is a stationary distribution for $Z$. In this direction, define first for every bounded measurable function $f:[a,b)\to \RR$ and every $q>0$
$$V^qf(x)\coloneqq 
\EE_x\left(\int_0^{\sigma(a,b)} f(X_t) {\mathcal E}_t \e^{ -(q+q_c)t} \dd t\right),\qquad x\in[a,b].$$
In particular, $V^0=U_{a,b}$ and the resolvent equation reads
$V^q=U_{a,b}-qU_{a,b} V^q$. With this notation at hand, and recalling from Lemma \ref{L7} that $\rho_{a,b} = q_c-1/r(a,b)$, we have
\begin{eqnarray*}
{\mathbf E}^{(a,b)}_{m_{a,b}}\left(\int_0^{\infty} \e^{-qt} f(Z_t) \dd t\right) &=&
\EE_{\nu_{a,b}}\left(\int_0^{\sigma(a,b)} \e^{-(q+q_c-1/r(a,b))t} f(X_t) h_{a,b}(X_t) {\mathcal E}_t \dd t\right) \\
&=& \ip{\nu_{a,b}}{V^{q-1/r(a,b)}(f h_{a,b})}\\
&=& \ip{\nu_{a,b}}{U_{a,b}(f h_{a,b})}- (q-1/r(a,b)) \ip{\nu_{a,b}}{U_{a,b} V^{q-1/r(a,b)}(f h_{a,b})}\\
&=& r(a,b)\ip{\nu_{a,b}}{f h_{a,b}}- (r(a,b)q-1) \ip{\nu_{a,b}}{V^{q-1/r(a,b)}(f h_{a,b})},
\end{eqnarray*}
where we used the resolvent equation at the third line, and that $U_{a,b}^*\nu_{a,b}=r(a,b)\nu_{a,b}$ at the fourth. 
This yields 
$${\mathbf E}^{(a,b)}_{m_{a,b}}\left(\int_0^{\infty} \e^{-qt} f(Z_t) \dd t\right) = q^{-1} \ip{\nu_{a,b}}{f h_{a,b}} 
= q^{-1} \ip{m_{a,b}}{f },$$
showing that indeed $m_{a,b}$ is a stationary law for $Z$.

Plainly, $Z$ inherits irreducibility from \eqref{e:kilirr}. Further, just as $X$, its trajectories have only negative jumps and increase between two consecutive jumps times. The existence of a stationary law then easily implies point recurrence, so that for all $x,y\in(a,b)$, in the obvious notation, 
\begin{eqnarray*}
1=\lim_{t\to \infty} {\mathbf P}^{(a,b)}_x(H_Z(y)\leq t) &=& \lim_{t\to \infty} \frac{1}{h_{a,b}(x)} \EE_x[ \mathcal{M}_{a,b}(t) , H(y)\leq t\wedge \sigma(a,b)]\\
&=& \frac{h_{a,b}(y)}{h_{a,b}(x)}  \lim_{t\to \infty} \EE_x[ \mathcal{E}_{H(y)} \e^{-\rho_{a,b}H(y)} , H(y)\leq t\wedge \sigma(a,b)],
\end{eqnarray*}
where the last equality stems from  Lemma \ref{L7} and Doob's optional sampling theorem. By monotone convergence, this proves that 
$$\EE_x\left({\mathcal E}_{H(y)} \e^{-\rho_{a,b}H(y)}, H(y) < \sigma(a,b)\right)= \frac{h_{a,b}(x)}{h_{a,b}(y)} .$$
\end{proof}

Proposition \ref{P1} enables us to compare the eigenvalues $\rho_{a,b}$ for nested good intervals, and also with the Malthus exponent $\lambda$.

\begin{lemma} \label{L9} We have:
\begin{enumerate}
\item[(i)] Let $(a,b)$ and $(a',b')$ two good intervals with $a<a'<b'<b$. Then 
$$\rho_{a',b'}< \rho_{a, b}< \lambda.$$
\item[(ii)]  $\lambda=\sup\{\rho_{a,b}: (a,b)\text{ is a good interval}\}$. 
\end{enumerate}
\end{lemma}
\begin{proof} 
(i) Since $\sigma(a',b')\leq \sigma(a,b)$ and the inequality is strict with positive $\PP_x$-probability for any $x\in [a',b']$, it follows from Proposition \ref{P1} that 
$$\EE_x\left({\mathcal E}_{H(x)} \e^{-\rho_{a',b'}H(x)}, H(x) < \sigma(a,b)\right)> 1 =
\EE_x\left({\mathcal E}_{H(x)} \e^{-\rho_{a,b}H(x)}, H(x) < \sigma(a,b)\right).$$
This forces $\rho_{a',b'} < \rho_{a,b}$. 
Similarly,
$$\EE_x({\mathcal E}_{H(x)}\e^{-\lambda H(x)}, H(x)<\sigma(a,b))<\EE_x({\mathcal E}_{H(x)}\e^{-\lambda H(x)}, H(x)<\infty)\leq 1,$$
where the second inequality follows from the  right-continuity of the Laplace transform $L_{x,x}$ and the definition \eqref{e:lambda} of the Malthus exponent. This yields $\rho_{a,b} < \lambda$.

(ii) Set $\bar\rho\coloneqq\sup\{\rho_{a,b}: (a,b)\text{ is a good interval}\}$. Then, by monotone convergence and Lemma \ref{L8}, we have 
\begin{eqnarray*}
& &\EE_x\left({\mathcal E}_{H(x)} \e^{-\bar\rho H(x)}, H(x)<\infty\right)\\
&=&\sup\{ \EE_x\left({\mathcal E}_{H(x)} \e^{-\bar\rho H(x)}, H(x)<\sigma(a,b)\right): (a,b)\text{ is a good interval}\}\\
&\leq &\sup\{ \EE_x\left({\mathcal E}_{H(x)} \e^{-\rho_{a,b} H(x)}, H(x)<\sigma(a,b)\right):  (a,b)\text{ is a good interval}\},
\end{eqnarray*}
showing that $L_{x,x}(\bar\rho)\leq 1$. Hence $\bar\rho\geq \lambda$, and the converse inequality follows from (i). 
\end{proof} 

\subsection{Letting the upper-boundary point go to $\infty$}

Next, for $a>0$, we write 
$$\sigma(a,\infty)=\lim_{b\to \infty}\sigma(a,b)=\inf\{t>0: X_t<a\}$$
for the first passage time below the level $a$ and state the main result of this sub-section:

\begin{proposition} \label{P8} Assume $\limsup_{x\to \infty} c(x)/x<\lambda$. Then the following hold:
\begin{enumerate}
\item[(i)]
There exists a good interval $(a,b)$ such that 
$\sup_{x\geq b} c(x)/x < \rho_{a,b}$. 

\item[(ii)] For every $b'\in(a,b)$ with $\sup_{x\geq b'} c(x)/x \leq \rho_{a,b}$, there exists $\lambda_a\in [\rho_{a,b}, \lambda)$ with 
$$\EE_{b'} \left( {\mathcal E}_{H(b')} \e^{-\lambda_{a} H(b')}, H(b')<\sigma(a, \infty)\right) =1.$$
\end{enumerate}
\end{proposition}
\begin{proof} (i) We first use Lemma \ref{L9} and the assumption that $\limsup_{x\to \infty} c(x)/x<\lambda$ to find a good interval $(\alpha, \beta)$  with $\limsup_{x\to \infty} c(x)/x<\rho_{\alpha, \beta}$. Then consider 
$$\gamma\coloneqq \sup\{x>0: c(x)/x\geq \rho_{\alpha,\beta}\}<\infty,$$ 
and choose a good interval $(a,b)$ with
$0<a<\alpha$ and $b>\beta\vee \gamma$. By Lemma \ref{L9}(i), we have then $\rho_{a,b}>\rho_{\alpha, \beta}$, and {\it a fortiori} $c(x)/x<\rho_{a,b}$ for all $x\geq b$. 

 (ii) Consider the convex and nonincreasing function $\Phi: \RR\to (0,\infty]$ defined by 
$$\Phi(q)\coloneqq \EE_{b'} \left( {\mathcal E}_{H(b')} \e^{-q H(b')}, H(b')<\sigma(a, \infty)\right).$$
Clearly from \eqref{e:defL} and \eqref{e:lambda}, $\Phi(\lambda)<L_{b',b'}(\lambda) \leq 1$, and we now check that $\Phi(\rho_{a,b})\in[1,\infty)$. This will entail our  claim, by taking for $\lambda_a$ the unique solution to $\Phi(q)=1$.

The lower bound should be plain from Proposition \ref{P1}; indeed\footnote{ We point out that this inequality may actually be an equality, as it may happen that the events $\{H(b')<\sigma(a, b)\}$ and $\{H(b')<\sigma(a, \infty)\}$ coincide $\PP_{b'}$-a.s. Specifically, irreducibility may fail for the process $X$ killed when crossing the level $a$, 
and it can occur that once the process $X$ becomes larger than $b$, it can no longer visit $b'$ without exiting first from $[a,\infty)$. Nonetheless, irreducibility will not be an issue in this subsection. }
$$\Phi(\rho_{a,b}) \geq \EE_{b'} \left( {\mathcal E}_{H(b')} \e^{-\rho_{a,b} H(b')}, H(b')<\sigma(a, b)\right)=1.$$
On the other hand, recall that the process started from $b'$ first stays in $[b',\infty)$ until it eventually  makes a jump across $b'$ at time $\sigma(b',\infty)$, and then stays in $(0,b')$ until it eventually hits $b'$ for the first time at the instant $H(b')$. 
Recall that $\sup_{x\geq b'} c(x)/x \leq  \rho_{a,b}$, which ensures that $ {\mathcal E}_{\sigma(b',\infty)} \e^{-\rho_{a,b} \sigma(b',\infty)}\leq 1$.
An application of the strong Markov property at time $\sigma(b',\infty)$ then enables us to bound $\Phi(\rho_{a,b})$ 
from above by 
\begin{eqnarray*}
& &\int_{[a,b')} \EE_x\left( {\mathcal E}_{H(b')} \e^{-\rho_{a,b} H(b')}, H(b')<\sigma(a, b)\right) \PP_{b'}(X_{\sigma(b',\infty)}\in\dd x, \sigma(b',\infty)<\infty)\\
&=&\int_{[a,b')} \frac{h_{a,b}(x)}{h_{a,b}(b')}\PP_{b'}(X_{\sigma(b',\infty)}\in\dd x, \sigma(b',\infty)<\infty),\end{eqnarray*}
where the equality is seen from Proposition \ref{P1}.
Since $h_{a,b}$ is bounded and $h_{a,b}(b')>0$, we conclude that $\Phi(\rho_{a,b})<\infty$. 
\end{proof}

In turn, Proposition \ref{P8}(ii) yields another remarkable martingale.
\begin{corollary} \label{C4} Under the same assumptions and notation as in Proposition \ref{P8},
introduce the function
$$g_a(x)\coloneqq \EE_x\left({\mathcal E}_{H(b')}\e^{-\lambda_a H(b')}, H(b')<\sigma(a,\infty)\right), \qquad x\geq a.
$$
Then the process
$${\mathcal M}_a(t)\coloneqq g_a(X_t) {\mathcal E}_t \e^{-\lambda_a t} \Indic{t<\sigma(a,\infty)}, \qquad t\geq 0$$
 is a $\PP_x$-martingale for all $x\in[a,\infty)$. 
\end{corollary} 
\begin{proof}
Since $g_a(b')=1$ by the very definition of $\lambda_a$ in Proposition \ref{P8}, the claim follows from the same argument as in the proof of Theorem 4.4 of \cite{BW}. \end{proof}

\subsection{Letting the lower-boundary point go to $0$}
We now rather fix the upper-boundary point $b$ and  let the lower-boundary point $a$ tend to $0$, and develop results similar to those of the preceding sub-section. Some arguments and statements need to be adapted to that case, other simply work just as well. Beware  in particular that the notation $\lambda^b$ below refers to a quantity that depends on the upper-boundary point $b$, and not to the Malthus exponent raised to the power $b$. 
Note also that due to the absence of positive jumps, the identity $\lim_{a\to 0}\sigma(a,b)=H(b)$ holds $\PP_x$-a.s. for all $x<b$. 

\begin{proposition} \label{P5} Assume  $\limsup_{x\to 0} c(x)/x<\lambda$. Then the following hold:
\begin{enumerate}
\item[(i)]
There exist a good interval  $(a, b)$ such that 
 $\sup_{x\leq a} c(x)/x <  \rho_{a,b}$. 

\item[(ii)] For every $a'\in(a,b)$ with $\sup_{x\leq a'} c(x)/x \leq   \rho_{a,b}$ and every $b''\in(a',b)$ sufficiently close to $b$, 
there exists then $\lambda^b<\lambda$ with 
$$\EE_{a'} \left( {\mathcal E}_{H(a')} \e^{-\lambda^b H(a')}, H(a')<H(b'')\right) \in(0,1].$$
\end{enumerate}
\end{proposition}

\begin{proof} (i) The argument is just the same as in Proposition \ref{P8}(i). 

(ii) The irreducibility  \eqref{e:kilirr} entails that $\PP_{a'}(H(a')<H(b''))>0$ provided that $b''\in(a',b)$ is chosen close enough to $b$. Then consider the convex and nonincreasing function $\Psi: \RR\to (0,\infty]$ defined by 
$$\Psi(q)\coloneqq \EE_{a'} \left( {\mathcal E}_{H(a')} \e^{-q H(a')}, H(a')<H(b'')\right).$$
Since $\rho_{a,b} <  \lambda$  by Lemma \ref{L9}, we may pick  $r\in
(\rho_{a,b},\lambda)$.
We shall check that $\Psi(r)<\infty$ and our  claim then follows. Indeed, if  actually $\Psi(r)\leq 1$ then we simply take $\lambda^b
=r$. Otherwise,  since we have always $\Psi(\lambda)<L_{a',a'}(\lambda)\leq 1$,  the equation $\Psi(q)=1$ has a unique solution 
 $\lambda^b\in(r, \lambda)$.

On the event $H(a')<H(b'')$, the process started from $a'$ first stays in $[a',b'']$ until  it makes a jump across $a'$ at time $\sigma(a',b'')$, and then stays in $(0,a')$ until it eventually hits $a'$ for the first time at the instant $H(a')$. 
Since 
$$\sup_{x\leq a'} c(x)/x \leq \rho_{a,b}<  r,$$ on the event that $H(a')<H(b'')$, there is the inequality 
$${\mathcal E}_{H(a')} \e^{-r H(a')} \leq {\mathcal E}_{\sigma(a',b'')} \e^{-r \sigma(a',b'')}.$$
We have therefore
\begin{eqnarray*}
\EE_{a'}\left( {\mathcal E}_{H(a')} \e^{-r H(a')}, H(a')<H(b'') \right) &\leq & \EE_{a'}\left({\mathcal E}_{\sigma(a',b'')} \e^{-r \sigma(a',b'')}, H(a')<H(b'')\right) \\
&=& \EE_{a'}\left(\sum_{t\geq 0} {\mathcal E}_{t} \e^{-r t} \Indic{t\leq \sigma(a',b''), X_t<a'}\right)\\
&= & 
\EE_{a'} \left(\int_0^{\sigma(a',b'')} {\mathcal E}_{t} \e^{-r t} \left(\int_0^{a'} \bar k(X_{t-},y) \dd y\right)\dd t\right)\\
&\leq & \| K\|
\EE_{a'}\left(\int_0^{ \sigma(a',b'')} {\mathcal E}_{t} \e^{-r t} \dd t\right),
\end{eqnarray*}
where the third line stems from the fact that the predictable compensator of the  jump process of $X$ is $\bar k(X_{t-},y) \dd y \dd t$, and on the last line, $\| K\|=\sup_{x>0} K(x)$ is the maximal jump rate.

We then write $\delta\coloneqq r-\rho_{a,b}>0$ and use the inequality
\begin{eqnarray*}
 \EE_{a'}\left({\mathcal E}_{t} \e^{-r t}, t<\sigma(a',b'')\right) &\leq& \frac{\e^{-\delta t}}{\min_{[a',b'']} h_{a, b}} 
 \EE_{a'}\left({\mathcal E}_{t} \e^{-\rho_{a,b}t} h_{a, b}(X_t), t<\sigma(a,b)\right)\\
 &=&  \e^{-\delta t} \frac{  h_{a, b}(a')}{\min_{[a',b'']} h_{a, b}},
 \end{eqnarray*}
 where the equality is seen from Lemma \ref{L7}. Since $\min_{[a',b'']} h_{a, b}>0$  by Proposition \ref{P2}, 
 we now get that
 $$\int_0^{\infty} \EE_{a'}\left( {\mathcal E}_{t} \e^{-r t} \Indic{ t<\sigma(a',b'')}\right)\dd t< \infty,$$
which entails our claim. 
\end{proof}

Proposition \ref{P5}(ii) enables us to repeat the argument for the proof of Lemma \ref{L2}, and this yields the following weak analog of Corollary \ref{C4}. 
\begin{corollary} \label{C6} Notation and assumptions are as in Proposition \ref{P5}.
 For $0<x< b''$, we consider
$$g^b(x)\coloneqq \EE_x\left({\mathcal E}_{H(a')}\e^{-\lambda^b H(a')}, H(a')<H(b'')\right).
$$
The process
$${\mathcal S}^b(t)\coloneqq g^b(X_t) {\mathcal E}_t \e^{-\lambda^b t} \Indic{t<H(b'')}, \qquad t\geq 0.$$
 is then a $\PP_x$-supermartingale for every $0<x<b''$. 
\end{corollary}

We have now completed all the preliminary steps needed to establish Theorem \ref{T2}.
\subsection{Proof of \eqref{e:condBW2}}
We first pick two good intervals  $(a,b)$ and $(a',b')$ with $0<a<a'<b'<b$ sufficiently large such that
$$\sup_{(0,a']\cup [b',\infty)} c(x)/x < \rho_{a,b}$$
(recall that this is indeed possible, thanks to Lemma \ref{L9} and the assumptions of Theorem \ref{T2}).
The irreducibility of the process killed when exiting from $[a,b]$ shows that provided we choose $b''\in(b',b)$ close enough to $b$, then $\PP_{b'}(H(a')<H(b''))>0$. Next, in the notation introduced in Proposition \ref{P8} and Proposition \ref{P5}, take any  $\max\{\rho_{a,b}, \lambda_a, \lambda^b\}< q < \lambda$. In particular, 
$$c(x)/x< q \ \text{ for all }\ x\in(0,a']\cup[b',\infty).$$

We shall prove that \eqref{e:condBW2} holds with $x=b'$. 
 We thus let the process start from $b'$ and split the excursion interval $(0, H(b'))$ at times $\sigma(b',\infty)$ and $\sigma(a',\infty)$. Recall that $\sigma(b',\infty)$ is the first instant when $X$ jumps across $b'$, so plainly $\sigma(b',\infty)\leq \sigma(a',\infty)$ $\PP_{b'}$-a.s.,  and  these two first-exit times may coincide with positive probability. Because $X$ stays in $[b',\infty)$ until time $\sigma(b',\infty)$, we have 
$${\mathcal E}_{\sigma(b',\infty)} \e^{-q \sigma(b',\infty)}\leq 1\qquad \PP_{b'}\text{-a.s.},$$
and we see from the strong Markov property that \eqref{e:condBW2} will follow from 
 \begin{equation} \label{e:condBW22}
 \sup_{x\leq b'} \EE_x\left( {\mathcal E}_{H(b')} \e^{-q H(b')}, H(b')<\infty\right) < \infty.
 \end{equation}
 
 First consider the case $x\leq a'$. The process started from $x$ stays in $(0,a']$ until time $H(a')$,  thus
 $${\mathcal E}_{H(a')} \e^{-q H(a')}\leq 1\qquad \PP_x\text{-a.s.},$$
and therefore we have, again from the strong Markov property,
$$  \EE_x\left( {\mathcal E}_{H(b')} \e^{-q H(b')}, H(b')<\infty\right) \leq  \EE_{a'}\left( {\mathcal E}_{H(b')} \e^{-q H(b')}, H(b')<\infty\right).$$
In the notation of Corollary \ref{C6}, Proposition  \ref{P5}(ii) ensures that  $g^b(a')\in(0,1]$, and since 
$b''$ has been chosen such that $\PP_{b'}(H(a')<H(b''))>0$, we have also $g^b(b')>0$. 
 We now deduce from the optional sampling theorem applied to the supermartingale ${\mathcal S}^b$ in Corollary \ref{C6} and the stopping time $H(b')<H(b'')$, that the right-hand side above is bounded by $g^b(a')/g^b(b')$, and hence
 \begin{equation} \label{e:clafin}
  \sup_{x\leq a'} \EE_x\left( {\mathcal E}_{H(b')} \e^{-q H(b')}, H(b')<\infty\right) \leq  \frac{g^b(a')}{g^b(b')}< \infty.
  \end{equation}
 
 Next, we consider the case $x\in(a',b']$, and distinguish whether the process exits from $[a',b']$ through the upper or the lower boundary.  
 On the one hand, Lemma \ref{L7} and the optional sampling theorem readily yield for every $x\in[a',b']$:
 $$\EE_x\left( {\mathcal E}_{H(b')} \e^{-q H(b')}, H(b')<\sigma(a,b)\right)\leq \frac{h_{a,b}(x)}{h_{a,b}(b')},$$
and {\it a fortiori} 
 $$\sup_{a'\leq x \leq b'} \EE_x\left( {\mathcal E}_{H(b')} \e^{-q H(b')}, H(b')\leq \sigma(a',b')\right)\leq \frac{\max_{[a',b']}h_{a,b}}{h_{a,b}(b')} <\infty. $$
 On the other hand, the same argument as in the proof of Proposition \ref{P5}(ii) shows that for every $x\in[a',b']$:
 \begin{eqnarray*}
\EE_{x}\left({\mathcal E}_{\sigma(a',b')} \e^{-q \sigma(a',b')}, \sigma(a',b')<H(b')\right) 
&=& \EE_{x}\left(\sum_{t\geq 0} {\mathcal E}_{t} \e^{-q t} \Indic{t\leq \sigma(a',b'), X_t<a'}\right)\\
&\leq & \| K\|
\EE_{x}\left(\int_0^{ \sigma(a',b')} {\mathcal E}_{t} \e^{-q t} \dd t\right)\\
&\leq & \frac{  \| K\| h_{a, b}(x)}{\delta \min_{[a',b']} h_{a, b}},
\end{eqnarray*}
where $\delta=q-\rho_{a',b'}>0$ and $\| K\|$ is the maximal jump rate. 
Combining this with \eqref{e:clafin} and the strong Markov property, we conclude that 
$$\sup_{a'\leq x \leq b'} \EE_x\left( {\mathcal E}_{H(b')} \e^{-q H(b')},\sigma(a',b')<  H(b')<\infty\right) <\infty.$$
This completes the proof of \eqref{e:condBW2}, and thus of Theorem \ref{T1}.

\subsection{Miscellaneous comments about \eqref{e:ccbis}} \label{s:MC}
\begin{enumerate}

\item We have argued in the Introduction that  \eqref{e:ccbis} yields the more explicit criterion \eqref{e:ccter} when $X$ is recurrent, and we now discuss simple conditions in terms of the growth and fragmentation rates that ensure recurrence. Since $X$ is irreducible and its trajectories have no positive jumps, point-recurrence can only fail if sample paths converge either to $0$ or to $\infty$ a.s., which can easily be impeded by Foster-type conditions (see for instance \cite{MT-book} and \cite{Hai-conv} for general references). Typically, consider the power function $f(x)=x^r$ for some $r>0$,  and assume that ${\mathcal G}f(x)\leq 0$ for all large enough $x$. That is, for some $x_{\infty}>0$:
\begin{equation}\label{e:mieux1}
r c(x) x^{r} + \int_{0}^{x} (y^r-x^r) y k(x,y)\dd y \leq 0 \qquad\text{ for all }x\geq x_{\infty}.
\end{equation}
Then the process $f(X_{t\wedge \sigma(x_{\infty},\infty)})$ is a $\PP_x$-supermartingale for every $x\geq x_{\infty}$, 
and it follows that $\PP_x(\lim_{t\to \infty} X_t=\infty)=0$.  Similarly, working now with negative powers and considering $g(x)=x^{-q}$ for some $q>0$, if for some $x_0>0$
\begin{equation}\label{e:mieux2}-q c(x) x^{-q} + \int_{0}^{x} (y^{-q}-x^{-q}) y k(x,y)\dd y \leq 0 \qquad\text{ for all }x\leq x_0,
\end{equation}
then  the process $g(X_{t\wedge H(x_0)})$ is a $\PP_x$-supermartingale for every $x\leq x_0$, 
and it follows that $\PP_x(\lim_{t\to \infty} X_t=0)=0$.  
We refer further to \citet{Bou}, who considered specifically positive recurrence for the family of piecewise deterministic Markov processes that arise in our setting; see notably Theorem 4 and also Remark 3 there.
Putting the pieces together and excluding implicitly the case of linear growth\footnote{Linear growth rate $c(x)=ax$ was already discussed in details in Section 6 of \cite{BW}}, we get that exponentially fast convergence \eqref{e:expconv} holds provided that  \eqref{e:ccter}, \eqref{e:mieux1} and \eqref{e:mieux2} are satisfied.

Recall that \citet{DoumGab} obtained conditions that ensure the existence of eigenelements and thus also  the Malthusian behavior \eqref{e:paradigm} by the general relative entropy method (however their approach does not yield exponential speed of convergence  \eqref{e:expconv}). The comparison of those conditions displays certain resemblance and  but also differences. For instance, (12) and (13)  in \cite{DoumGab} can be loosely related to \eqref{e:mieux1} and \eqref{e:mieux2} here;  we do not need here to make assumptions such as (5,7,10) in \cite{DoumGab}; on the other hand \cite{DoumGab} also covers the situation where the growth rate $c$ does not fulfill \eqref{e:c-bound}. 

\item If we assume self-similarity of the fragmentation kernel, that is
$$k(x,y)= x^{-1} K(x)p(y/x), \qquad 0<y<x,$$
where $p\in L^1_+([0,1])$  with $\int_0^1p(u) u \dd u=1$, then \eqref{e:mieux1} and \eqref{e:mieux2} translate respectively  into 
\begin{equation}\label{e:mieux1'}
\frac{c(x)}{xK(x)} \leq  r^{-1}\int_0^1(1-u^r)u p(u)\dd u \qquad\text{ for all }x\geq x_{\infty}
\end{equation}
and 
\begin{equation}\label{e:mieux2'}
\frac{c(x)}{xK(x)} \geq  q^{-1}\int_0^1(u^{-q}-1)u p(u)\dd u \qquad\text{ for all }x\leq x_0.
\end{equation}
Putting pieces together, we thus see that in the self-similar case, exponentially fast convergence \eqref{e:expconv} holds whenever \eqref{e:ccter}, \eqref{e:mieux1'} and \eqref{e:mieux2'} are fulfilled.

This should be compared with Theorem 1.11 of  \citet{BCG-fine} in which the existence of a spectral gap is asserted under more stringent conditions. We stress however that \cite{BCG-fine} also cover cases where the growth rate $c$ does not fulfill \eqref{e:c-bound}. If we further assume that the total fragmentation rate is constant, say $K(x)\equiv 1$, then we are in the setting of Section 7 of \cite{BW} (one says that the fragmentation rate is homogeneous), and Theorem \ref{T2} also improves Proposition 7.1 in \cite{BW}.

\item We point out that exponential speed of convergence \eqref{e:expconv} may hold without \eqref{e:ccbis}. For instance, in the case of linear growth rate $c(x)=ax$, Section 6 in \cite{BW} discusses situations where $\lambda=a$
and nonetheless \eqref{e:expconv} takes place.

\item We also recall from Corollary 4.5 and Lemma 4.6 in \cite{BW}, that when \eqref{e:ccbis} holds, then $h$ 
is an eigenfunction for the eigenvalue $\lambda$ of the growth-fragmentation operator ${\mathcal A}$ (in particular $h$ belongs to the domain of ${\mathcal A}$), and further the function $x\mapsto h(x)/x$ is continuous and bounded.

\end{enumerate}

\bibliography{AGF}

\end{document}